\documentclass{amsproc}

\usepackage{amsmath}
\usepackage{latexsym}
\usepackage{amsfonts}
\usepackage{amssymb}
\usepackage{color}

\newtheorem{thm}{Theorem}[section]
\newtheorem{lem}[thm]{Lemma}
\newtheorem{prop}[thm]{Proposition}
\newtheorem{cor}[thm]{Corollary}

\theoremstyle{definition}
\newtheorem{defn}[thm]{Definition}

\theoremstyle{remark}
\newtheorem{rem}[thm]{Remark}

\numberwithin{equation}{section}

\newcommand{\sca}{\mathcal{S}\mathcal{C}}
\newcommand{\fr}{\mathcal{F}}
\newcommand{\Id}{\operatorname{Id}}
\newcommand{\co}{\operatorname{co}}

\newcommand{\aff}{\operatorname{aff}}

\newcommand{\Tr}{\operatorname{Tr}}
\newcommand{\diag}{\operatorname{diag}}
\newcommand{\ip}[2]{\langle#1,#2\rangle}

\renewcommand{\emptyset}{\varnothing}

\newcommand{\R}{\ensuremath{\mathbb R}}    % Reelle Zahlen
\newcommand{\N}{\ensuremath{\mathbb N}}    % Nat"urliche Zahlen
\newcommand{\calF}{\mathcal F}         
         
\newcommand{\calH}{\mathcal H}         
\newcommand{\calI}{\mathcal I}

\newcommand{\la}{\lambda}

\newcommand{\vphi}{\varphi}

\DeclareMathOperator*{\linspan}{span}
\renewcommand{\ker}{\operatorname{ker}}

\newcommand{\Sra}{\Rightarrow}

\newcommand{\Sla}{\Leftarrow}

\newcommand{\Llra}{\Longleftrightarrow}
\newcommand{\Slra}{\Leftrightarrow}

\newcommand{\<}{\langle}

\renewcommand{\>}{\rangle}

\begin{document}

\title{Scalable Frames and Convex Geometry}

% Information for first author
\author{Gitta Kutyniok}
\address{Technische Universit\"at Berlin, Institut f\"ur Mathematik, Strasse des 17. Juni 136, 10623 Berlin, Germany}
\email{kutyniok@math.tu-berlin.de}
%\thanks{TBA.}
\author{Kasso A. Okoudjou}
\address{University of Maryland, Department of Mathematics, College Park, MD 20742 USA}
\email{kasso@math.umd.edu}
%\thanks{TBA.}
\author{ Friedrich Philipp}
\address{Technische Universit\"at Berlin, Institut f\"ur Mathematik, Strasse des 17. Juni 136, 10623 Berlin, Germany}
\email{philipp@math.tu-berlin.de}
%\thanks{TBA.}

%\author{Author One}
%%    Address of record for the research reported here
%\address{Department of Mathematics, Louisiana State University, Baton
%Rouge, Louisiana 70803}
%%    Current address
%\curraddr{Department of Mathematics and Statistics,
%Case Western Reserve University, Cleveland, Ohio 43403}
%\email{xyz@math.university.edu}
%%    \thanks will become a 1st page footnote.
%\thanks{The first author was supported in part by NSF Grant \#000000.}
%
%%    Information for second author
%\author{Author Two}
%\address{Mathematical Research Section, School of Mathematical Sciences,
%Australian National University, Canberra ACT 2601, Australia}
%\email{two@maths.univ.edu.au}
%\thanks{Support information for the second author.}

%    General info
\subjclass[2010]{Primary 42C15, 52B11; Secondary 15A03, 65F08}
\date{}

%\dedicatory{This paper is dedicated to our advisors.}

\keywords{Scalable frames, tight frames, preconditioning, Farkas's lemma}

\begin{abstract}
The recently introduced and characterized scalable frames can be considered as those frames which allow for
perfect preconditioning in the sense that the frame vectors can be rescaled to yield a tight frame. In this
paper we define $m$-scalability, a refinement of scalability based on the number of non-zero weights used in
the rescaling process, and study the connection between this notion and elements from convex geometry.
Finally, we provide results on the topology of scalable frames. In particular, we prove that the set of
scalable frames with ``small'' redundancy is nowhere dense in the set of frames.
\end{abstract}

\maketitle

\section{Introduction}\label{sec:intro}
Frame theory is nowadays a standard methodology in applied mathematics and engineering. The key advantage of
frames over orthonormal bases is the fact that frames are allowed to be redundant, yet provide stable
decompositions. This is a crucial fact, for instance, for applications which require robustness against
noise or erasures, or which require a sparse decomposition (cf. \cite{ck12}).

Tight frames provide optimal stability, since these systems satisfy the Parseval equality up to a constant.
Formulated in the language of numerical linear algebra, a tight frame is perfectly conditioned, since the
condition number of its analysis operator is one.
Thus, one key question is the following: Given a frame $\Phi = \{\varphi_k\}_{k=1}^{M}\subset\R^N$, $M\ge N$,
say, can the frame vectors $\varphi_k$ be modified so that the resulting system forms a tight frame? Again
in numerical linear algebra terms, this question can be regarded as a request for perfect preconditioning \cite{bb,c}.
Since a frame is typically designed to accommodate certain requirements of an application, this modification
process should be as careful as possible in order not to change the properties of the system too drastically.

One recently considered approach consists in multiplying each frame vector by a scalar/a weight. Notice that this process does  not even
disturb sparse decomposition properties at all, hence it might be considered `minimally invasive'. The formal
definition was given in \cite{kopt12} by the authors and E.K.\ Tuley (see also \cite{kop13a}). In that paper,
a frame, for which scalars exist so that  the scaled frame forms a tight frame, was coined {\em scalable frame}.
Moreover, in the infinite dimensional situation, various equivalent conditions for scalability were provided,
and in the finite dimensional situation, a very intuitive geometric characterization was proven. In fact, this
characterization showed that a frame is non-scalable, if the frame vectors do not spread `too much' in the space.
This seems to indicate that there exist relations to convex geometry.

Scalable frames were then also investigated in the papers \cite{cklmns12} and \cite{cc13}. In \cite{cklmns12}, the
authors analyzed the problem by making use of the properties of so-called diagram vectors \cite{hklw}, whereas
\cite{cc13} gives a detailed insight into the set of weights which can be used for scaling.

The contribution of the present paper is three-fold. First, we refine the definition of scalability by calling
a (scalable) frame $m$-scalable, if at most $m$ non-zero weights can be used for the scaling. Second, we establish
a link to convex geometry. More precisely, we prove that this refinement leads to a reformulation of the
scalability question in terms of the properties of certain polytopes associated to a nonlinear transformation of
the frame vectors. This nonlinear transformation is related but not equivalent to the diagram vectors used in the
results obtained in \cite{cklmns12}. Using this reformulation, we establish new characterizations of scalable
frames using convex geometry, namely convex polytopes. Third, we investigate the topological properties of the set
of scalable frames. In particular, we prove that in the set of frames in $\R^N$ with $M$ frame vectors the set of
scalable frames is nowhere dense if $M < N(N+1)/2$. We wish to mention, that the results stated and proved in this
paper were before announced in \cite{kop-spie}.

The paper is organized as follows. In Section \ref{sec:prelim}, we introduce the required notions with respect to
frames and their ($m$-)scalability as well as state some basic results. Section \ref{sec2} is devoted to establishing
the link to convex geometry and derive novel characterizations of scalable frames using this theory. Finally, in
Section \ref{sec3}, we study the topology of the set of scalable frames.

%**********************************************************************************************************************
%**********************************************************************************************************************
%**********************************************************************************************************************

\section{Preliminaries}
\label{sec:prelim}

First of all, let us fix some notation. If $X$ is any set whose elements are indexed by $x_j$, $j\in J$, and $I\subset J$,
we define $X_I := \{x_i : i\in I\}$. Moreover, for the set $\{1,\ldots,n\}$, $n\in\N$, we write $[n]$.

A set $\Phi = \{\varphi_k\}_{k=1}^{M}\subset\R^N$, $M\ge N$ is called a {\em frame}, if there exist positive constants $A$
and $B$ such that
\begin{equation}\label{frameineq}
A\|x\|^2\,\le\,\sum_{k=1}^{M}|\ip{x}{\varphi_k}|^{2}\,\le\,B \|x\|^2
\end{equation}
holds for all $x\in\R^N$. Constants $A$ and $B$ as in \eqref{frameineq} are called \emph{frame bounds} of $\Phi$.
The frame $\Phi$ is called {\em tight} if $A = B$ is possible in \eqref{frameineq}. In this case we have
$A = \tfrac 1 N\sum_{=1}^{M}\|\varphi_k\|^{2}$. A tight frame with $A=B=1$ in \eqref{frameineq} is called
\emph{Parseval frame}.

We will sometimes identify a frame $\Phi = \{\varphi_{k}\}_{k=1}^{M}\subset\R^N$ with the $N\times M$ matrix
whose $k$th column is the vector $\varphi_k$. This matrix is called the {\em synthesis operator} of the frame.
The adjoint $\Phi^T$ of $\Phi$ is called the \emph{analysis operator}. Using the analysis operator, the relation
\eqref{frameineq} reads
$$
A\|x\|^2\,\le\,\|\Phi^Tx\|^2\,\le\,B \|x\|^2.
$$
Hence, a frame $\Phi$ is tight if and only if some multiple of $\Phi^T$ is an isometry. The set of frames
for $\R^N$ with $M$ elements will be denoted by $\fr(M,N)$. We say that a frame $\Phi\in\fr(M,N)$ is {\em degenerate} if
one of its frame vectors is the zero-vector. If $\mathcal{X}(M,N)$ is a set of frames in $\fr(M,N)$, we denote by
$\mathcal{X}^*(M,N)$ the set of the non-degenerate frames in $\mathcal{X}(M,N)$. For example, $\fr^*(M,N)$ is the set
of non-degenerate frames in $\fr(M,N)$. For more details on frames, we refer the reader to \cite{co03,ck12}

Let us recall the following definition from \cite[Definition 2.1]{kopt12}.

\begin{defn}\label{def0} 
A frame $\Phi = \{\varphi_k\}_{k=1}^M$ for $\R^N$ is called {\em scalable}, respectively, {\em strictly scalable}, if there exist nonnegative, respectively, positive,  scalars $c_1,\ldots,c_M\in\R$ such that $\{c_k\varphi_k\}_{k=1}^M$ is a tight frame for $\R^N$. The set of  scalable, respectively, strictly scalable,  frames in $\fr(M,N)$ is denoted by $\sca(M,N)$, respectively, $\sca_{+}(M,N)$.
\end{defn}

In order to gain a better understanding of the structure of scalable frames we refine the definition of scalability.

\begin{defn}\label{def1} 
Let $M,N,m\in\N$ be given such that $N\leq m \leq M$. A frame $\Phi=\{\varphi_k\}_{k=1}^{M} \in \fr(M,N)$  is said to be \emph{$m$-scalable}, respectively, \emph{strictly $m$-scalable}, if there exists a subset $I\subseteq [M]$, $\#I=m$, such that $\Phi_I$ is a scalable frame, respectively, a strictly scalable frame  for $\R^N$.  
We denote the set of $m$-scalable frames, respectively, strictly $m$-scalable frames  in $\fr(M,N)$ by   $\sca(M,N,m)$, respectively, $\sca_{+}(M,N,m).$ 
\end{defn}

It is easily seen that for $m\leq m'$ we have that $\sca(M,N, m)\subset \sca(M,N, m')$. Therefore,
$$
\sca(M,N) = \sca(M,N,M) = \bigcup_{m=N}^{M}\sca(M,N, m).
$$
In the sequel, if no confusion can arise, we often only write $\calF$, $\sca$, $\sca_{+}$, $\sca(m)$, and $\sca_{+}(m)$
instead of $\sca(M,N)$, $\sca_{+}(M,N)$, $\sca(M,N,m)$, and $\sca_{+}(M,N,m)$, respectively. The notations $\calF^*$,
$\sca^*$, $\sca_{+}^*$, $\sca(m)^*$, and $\sca_{+}(m)^*$ are to be read analogously.

Note that for a frame $\Phi\in\fr$ to be $m$-scalable it is necessary that $m\geq N$.
In addition, $\Phi\in\sca(M,N)$ holds if and only if $T(\Phi)\in\sca(M,N)$ holds for one (and hence for all) orthogonal
transformation(s) $T$ on $\R^N$; cf.\ \cite[Corollary 2.6]{kopt12}.

If $M\geq N$, we have $\Phi \in \sca(M,N,N)$ if and only if $\Phi$ contains an orthogonal basis of $\R^N$. This completely
characterizes the set $\sca(M,N,N)$ of $N$-scalable frames for $\R^N$ consisting of $M$ vectors. For frames with $M=N+1$
vectors in $\R^N$ we have the following result:

\begin{prop}\label{sframedplus1}
Let $N\geq 2$ and $\Phi=\{\varphi_k\}_{k=1}^{N+1}\in\fr^*$ with $\varphi_k\neq\pm\varphi_\ell$ for $k\neq\ell$. If
$\Phi\in\sca_+(N+1,N,N)$ then $\Phi\notin\sca_+(N+1,N)$.
\end{prop}

\begin{proof}
If $\Phi\in \sca_{+}(N+1, N, N)$, then $\Phi$ must contain an orthogonal basis. By applying some orthogonal transformation
and rescaling the frame vectors, we can assume without loss of generality that $\{\varphi_k\}_{k=1}^N=\{e_k\}_{k=1}^N$ is
the standard orthonormal basis of $\R^N$, and that $\varphi_{N+1}\neq \pm e_k$ for each $k=1, 2, \hdots, N$, with
$\|\varphi_{N+1}\|=1$. Thus, $\Phi$ can be written as $\Phi = \begin{bmatrix}\Id_N & \varphi_{N+1}\end{bmatrix}$, where
$\Id_N$ is the $N\times N$ identity matrix.

Assume that there exists $\{\lambda_k\}_{k=1}^{N+1}\subset (0,\infty)$ such that $\widetilde{\Phi}=\{\lambda_k \varphi_k\}_{k=1}^{N+1}$
is a tight frame, i.e.\ $\widetilde{\Phi}\widetilde{\Phi}^{T} = A\Id_N$. Using a block multiplication this equation can be rewritten as
$$
\Lambda + \lambda_{N+1}^2\varphi_{N+1}\varphi_{N+1}^{T} = A\Id_{N}
$$
where $\Lambda = \diag(\lambda_k^2)$ is the $N\times N$ diagonal matrix with $\lambda_k^2$, $k=1,\hdots,N$, on its diagonal. Consequently,
\begin{align*}
&\lambda_k^2 + \lambda_{N+1}^2\varphi_{N+1,k}^2 = A\;\text{ for }k=1,\hdots,N\quad\text{and}\\
&\lambda_{N+1}^2\varphi_{N+1,\ell}\varphi_{N+1,k} = 0\,\text{ for }k\neq\ell.
\end{align*}
But $\lambda_{N+1}>0$ and so all but one entry in $\varphi_{N+1}$ vanish. Since $\varphi_{N+1}$ is a unit norm vector, we see that
$\varphi_{N+1}=\pm e_k$ for some $k\in [N]$ which is contrary to the assumption, so $\Phi$ cannot be strictly $(N+1)$-scalable.
\end{proof}

%**********************************************************************************************************************
%**********************************************************************************************************************
%**********************************************************************************************************************

\section{Scalable Frames and Convex Polytopes}
\label{sec2}

Our characterizations of $m$-scalable frames will be stated in terms of certain convex polytopes and, more generally,
using tools from convex geometry. Therefore, we collect below some key facts and properties needed to state and prove our
results. For a detailed treatment of convex geometry we refer to \cite{matou02, stwi70, web94}.

%**********************************************************************************************************************
%**********************************************************************************************************************

\subsection{Background on Convex Geometry}
\label{subsec2.1}

In this subsection, let $E$ be a real linear space, and let $X =\{x_i\}_{i=1}^{M}$ be a finite set in $E$. The
\emph{convex hull generated by $X$} is the  compact convex subset of $E$ defined by
$$
\co(X) := \left\{ \sum_{i=1}^{M} \alpha_{i} x_i:  \alpha_i \geq 0, \, \sum_{i=1}^{M}\alpha_i=1\right\}.
$$
The \emph{affine hull generated by $X$} is defined by
$$
\aff(X) := \left\{ \sum_{i=1}^{M} \alpha_{i} x_i:   \, \sum_{i=1}^{M}\alpha_i=1\right\}.
$$
Hence, we have $\co(X)\subset\aff(X)$. Recall that for fixed $a\in\aff(X)$, the set
$$
V(X) := \aff(X) - a = \{y-a : y\in\aff(X)\}
$$
is a subspace of $E$ (which is independent of $a\in\aff(X)$) and that one defines
$$
\dim X := \dim \co(X) := \dim\aff(X) := \dim V(X).
$$
We shall use Carath\'eodory's Theorem for convex polytopes (see, e.g., \cite[Theorem 2.2.12]{stwi70}) in deciding
whether a frame is scalable:

\begin{thm}[Carath\'eodory]\label{t:cara}
Let $X = \{x_1,\ldots,x_k\}$ be a finite subset of $E$  with $d := \dim X$. Then for each $x\in\co(X)$ there exists
$I\subset [k]$ with $\#I = d+1$ such that $x\in\co(X_I)$.
\end{thm}

The \emph{relative interior} of the polytope $\co(X)$ denoted by $ri\co(X)$, is the interior of $\co(X)$ in the
topology induced by $\aff(X)$. We have that $ri\co(X)\neq\emptyset$ as long as $\#X\ge 2$; cf.\ \cite[Lemma 3.2.8]{stwi70}.
Furthermore,
$$
ri\co(X) = \left\{\sum_{i=1}^{M}\lambda_i x_i: \lambda_i>0, \sum_{i=1}^{M}\lambda_i = 1\right\},
$$
see \cite[Theorem 2.3.7]{web94}. Moreover, the interior of $\co(X)$ in $E$ is non-empty if and only if $\aff(X)= E$.

The following lemma characterizes $\dim X$ in terms of $\dim\linspan X$.

\begin{lem}\label{l:affine}
Let $X$ be a finite set of points in $E$.  Put $m := \dim\linspan X$. Then $\dim X\,\in\,\{m-1,m\}. $
Moreover, the following statements are equivalent:
\begin{enumerate}
\item[{\rm (i)}]   $\dim X = m-1$.
\item[{\rm (ii)}]  For all linearly independent $X'\subset X$ with $\dim\linspan X' = m$ we have $
X\setminus X'\subset\aff(X').$
\item[{\rm (iii)}] For some linearly independent $X'\subset X$ with $\dim\linspan X' = m$ we have $
X\setminus X'\subset\aff(X'). $
\end{enumerate}
\end{lem}

\begin{proof}
Let $X = \{x_1,\ldots,x_k\}$. First of all, we observe that for a linearly independent set $X' = \{x_{i_1},\ldots,x_{i_m}\}$
as in (ii) or (iii) we have
$$
\dim V(X') = \dim\linspan\{x_{i_l} -  x_{i_1} : l=2,\ldots,m\} = m - 1.
$$
Therefore, $V(X')\subset V(X)\subset\linspan X$ implies $m-1\leq\dim X\leq m$. Let us now prove the moreover-part of the lemma.

(i)$\Sra$(ii). Assume that $\dim X = m-1$ and let $X' = \{x_{i_1},\ldots,x_{i_m}\}$ be a linearly independent set as in
(ii). From $\dim V(X) = \dim X = m-1$ we obtain $V(X) = V(X')$. Therefore, for each $x_j\in X\setminus X'$ there exist
$\mu_2,\ldots,\mu_m\in\R$ such that
$$
x_j - x_{i_1} = \sum_{i=2}^m\mu_i(x_i - x_{i_1}) = \sum_{i=2}^m\mu_ix_i - \left(\sum_{i=2}^m\mu_i\right)x_{i_1}.
$$
And this implies
$$
x_j = \left(1 - \sum_{i=2}^m\mu_i\right)x_{i_1} + \sum_{i=2}^m\mu_ix_i\,\in\,\aff(X').
$$

(ii)$\Sra$(iii). This is obvious.

(iii)$\Sra$(i). Let $X' = \{x_{i_1},\ldots,x_{i_m}\}$ be a linearly independent set as in (iii). If $x\in X\setminus X'$,
then we have $x\in\aff(X')$ by (iii). Consequently, there exist $\lambda_1,\ldots,\lambda_m\in\R$ with $\sum_{l=1}^m\lambda_l = 1$
such that $x = \sum_{l=1}^m\lambda_l x_{i_l}$. Hence, we obtain
$$
x - x_{i_1} = \sum_{l=1}^m\lambda_l x_{i_l} - \left(\sum_{l=1}^m\lambda_l\right)x_{i_1} = \sum_{l=1}^m\lambda_l(x_{i_l} - x_{i_1})\,\in\,V(X').
$$
This implies $V(X) = V(X')$ and hence (i).
\end{proof}

In the sequel we will have to deal with a special case of the situation in Lemma~\ref{l:affine}, where $X$ is a set of rank-one
orthogonal projections acting on a real or complex Hilbert space $\calH$. In this case, $E$ is the set consisting of the
selfadjoint operators in $\calH$ which is a real linear space.

\begin{cor}\label{c:proj}
Let $X$ be a finite set consisting of rank-one orthogonal projections acting on a Hilbert space $\calH$. Then we have
$$
\dim X = \dim\linspan X - 1.
$$
\end{cor}

\begin{proof}
Let $X = \{P_1,\ldots,P_k\}$, $m := \dim\linspan X$, and let $X'\subset X$ be a linearly independent subset of $X$ such that
$\dim\linspan X' = m$. Without loss of generality assume that $X' = \{P_1,\ldots,P_m\}$. Let $j\in\{m+1,\ldots,k\}$. Then
there exist $\lambda_1,\ldots,\lambda_m\in\R$ such that $P_j = \sum_{i=1}^m\lambda_iP_i$. This implies
$$
1 = \Tr P_j = \Tr\left(\sum_{i=1}^m\lambda_iP_i\right) = \sum_{i=1}^m\lambda_i\Tr(P_i) = \sum_{i=1}^m\lambda_i,
$$
which shows that $P_j\in\aff(X')$. The statement now follows from Lemma \ref{l:affine}.
\end{proof}

%**********************************************************************************************************************
%**********************************************************************************************************************

\subsection{Scalability in Terms of Convex Combinations of Rank-One Matrices}
\label{subsec2.2}

Here, and for the rest of this paper, for a frame $\Phi = \{\varphi_i\}_{i=1}^M$ in $\fr(M,N)$ we set
$$X_\Phi := \{\varphi_i\varphi_i^T : i\in [M]\}.$$ This is a subset of the space of all real symmetric $N\times N$-matrices
which we shall denote by $S_N$. We shall also denote the set of positive multiples of the identity by
$\boldsymbol I_+ := \{\alpha\Id_N : \alpha > 0\}$.

\begin{prop}\label{p:charac_conv}
For a frame $\Phi\in\fr(M,N)$ the following statements are equivalent:
\begin{enumerate}
\item[{\rm (i)}]  $\Phi$ is  scalable, respectively, strictly scalable.
\item[{\rm (ii)}] $\boldsymbol I_+\cap\co(X_\Phi)\neq\emptyset$, respectively,  $\boldsymbol I_+\cap\,ri \co(X_\Phi)\neq\emptyset$.
\end{enumerate}
\end{prop}

\begin{proof}
Assume that the frame $\Phi = \{\varphi_i\}_{i=1}^M$ is scalable. Then there exist non-negative scalars $c_1,\ldots,c_M$ such that
$$
\sum_{i=1}^Mc_i\varphi_i\varphi_i^T = \Id.
$$
Put $\alpha := \sum_{i=1}^Mc_i$. Then $\alpha > 0$ and with $\lambda_i := \alpha^{-1}c_i$ we have
$$
\sum_{i=1}^M\lambda_i\varphi_i\varphi_i^T = \alpha^{-1}\Id\quad\text{and}\quad\sum_{i=1}^M\lambda_i = 1.
$$
Hence $\alpha^{-1}\Id\in\co(X_\Phi)$. The converse direction is obvious.
\end{proof}

As pointed out earlier, for $m\leq m'$ we have  $\sca(m)\subset\sca(m')$.  Given $\Phi\in \sca(M,N) = \sca(M)$,
there exists $m\leq M$ such that such that $\Phi\in\sca(m)$, e.g., we can always take $m=M$.  However, the next result gives a ``canonical'' integer $m=m_\Phi$ that  is in a way ``optimal''. 

\begin{prop}\label{p:less_vectors}
For a frame $\Phi=\{\varphi_k\}_{k=1}^{M}\in\fr$, put $m = m_\Phi := \dim\linspan X_\Phi$. Then the following statements are equivalent:
\begin{enumerate}
\item[{\rm (i)}]  $\Phi$ is scalable.
\item[{\rm (ii)}] $\Phi$ is $m$-scalable.
\end{enumerate}
\end{prop}
\begin{proof}
Clearly, (ii) implies (i). Conversely, let $\Phi = \{\varphi_i\}_{i=1}^M$ be scalable. After possibly removing zero vectors from the
frame and thereby reducing $M$ (which does not affect the value of $m$), we may assume that $\Phi$ is unit-norm. By Proposition
\ref{p:charac_conv}, there exists $\alpha > 0$ such that $\alpha\Id_N\in\co(X_\Phi)$. Therefore, from Theorem \ref{t:cara} it
follows that there exists $I\subset [M]$ with $\#I = \dim X_\Phi + 1$ such that $\alpha\Id_N\in\co(X_{\Phi_I})$. Hence, $\Phi_I$
is scalable by Proposition \ref{p:charac_conv}. And since $\dim X_\Phi = \dim\linspan X_\Phi - 1$ by Corollary \ref{c:proj}, the claim follows.
\end{proof}

As $X_\Phi\subset S_N$ and $\dim S_N = N(N+1)/2$, we immediately obtain the following corollary.

\begin{cor}\label{highredund}
For $M\ge N(N+1)/2$ we have
$$
\sca(M,N) = \sca\left(M,N,\frac{N(N+1)}{2}\right).
$$
\end{cor}

%**********************************************************************************************************************
%**********************************************************************************************************************

\subsection{Convex Polytopes Associated with $\mathbf{m}$-Scalable Frames}
\label{subsec2.3}

Let $\Phi = \{\vphi_k\}_{k=1}^M$ be a frame for $\R^N$. Then the analysis operator of the scaled frame $\{c_k\vphi_k\}_{k=1}^M$
is given by $C\Phi^T$, where $C$ is the diagonal matrix with the values $c_k$ on its diagonal. Hence, the frame $\Phi$ is
scalable if and only if
\begin{equation}\label{e:scfrm}
\Phi C^{2} \Phi^{T} = A\Id_N,%\tfrac{\sum_{k\in I}c_{k}^{2}\|\varphi_{k}\|^{2}}{N}\Id_{N}
\end{equation}
where $A > 0$. Similarly, $\Phi$ is $m$-scalable if and only if \eqref{e:scfrm} holds with $C = \diag(c)$, where
$c\in [0,\infty)^M$ such that $\|c\|_0\le m$. Here, we used the so-called ``zero-norm'' (which is in fact not a norm),
defined by
$$
\|x\|_0 := \#\{k\in [n] : x_k\neq 0\},\quad x\in\R^n.
$$
Comparing corresponding entries from left- and right-hand sides of \eqref{e:scfrm}, it is seen that for a frame to be
$m$-scalable it is necessary and sufficient that there exists a vector $u = (c_{1}^{2},c_{2}^{2},\ldots,c_{M}^{2})^{T}$
with $\|u\|_0\le m$ which is a solution of the following linear system of $\tfrac{N(N+1)}{2}$ equations in $M$ unknowns:
\begin{equation}\label{nxm}
\begin{cases}
\sum\limits_{j=1}^M\vphi_j(k)^2y_j = A &\text{for }k=1,\ldots,N,\\
\sum\limits_{j=1}^M\vphi_j(\ell)\vphi_j(k)y_j = 0 &\text{for }\ell,k=1,\ldots,N,\,k > \ell.
\end{cases}
\end{equation}

Subtraction of equations in the first system in \eqref{nxm} leads to the new {\em homogeneous} linear system
\begin{equation}\label{nxm2}
\begin{cases}
\sum\limits_{j=1}^M\big(\vphi_j(1)^2 - \vphi_j(k)^2\big)y_j = 0 &\text{for }k=2,\ldots,N,\\
\sum\limits_{j=1}^M\vphi_j(\ell)\vphi_j(k)y_j = 0 &\text{for }\ell,k=1,\ldots,N,\,k > \ell.
\end{cases}
\end{equation}
It is not hard to see that we have not lost information in the last step, hence $\Phi$ is $m$-scalable if and only
if there exists a nonnegative vector $u\in\R^M$ with $\|u\|_0\le m$ which is a solution to \eqref{nxm2}. In matrix form,
\eqref{nxm2} reads
\begin{equation*}%\label{matrixu}
F(\Phi)u = 0,
\end{equation*}
where the $(N-1)(N+2)/2\times M$ matrix $F(\Phi)$ is given by
\begin{equation*}%\label{Nreducmat}
F(\Phi) = \begin{pmatrix} F(\varphi_{1}) & F(\varphi_{2})  & \hdots & F(\varphi_{M})  \end{pmatrix},
\end{equation*}
where $F : \R^N\to\R^{d}$, $d := (N-1)(N+2)/2$, is defined by
\begin{equation*}\label{fctF}
F(x)=
\begin{pmatrix}
F_{0}(x)\\ F_{1}(x)\\ \vdots\\ F_{N-1}(x)
\end{pmatrix},
\qquad
F_{0}(x) =
\begin{pmatrix} x_{1}^{2}-x_{2}^{2}\\ x_{1}^{2}-x_{3}^{2} \\ \vdots \\ x_{1}^{2}-x_{N}^{2}\end{pmatrix},
\qquad
F_{k}(x) =
\begin{pmatrix}x_{k}x_{k+1}\\ x_{k}x_{k+2}\\ \vdots \\ x_{k}x_{N}\end{pmatrix},
\end{equation*}
and $F_{0}(x)\in \R^{N-1}$, $F_{k}(x) \in \R^{N-k}$, $k=1, 2, \hdots, N-1$.

Summarizing, we have just proved the following proposition.

\begin{prop}\label{p:F}
A frame $\Phi$ for $\R^N$ is  $m$-scalable, respectively,  strictly $m$-scalable  if and only if there exists a nonnegative
$u\in\ker F(\Phi)\setminus\{0\}$ with $\|u\|_0\le m$, respectively,  $\|u\|_0 = m$.
\end{prop}

We will now utilize the above reformulation to characterize $m$-scalable frames in terms of the properties of
convex polytopes of the type $\co(F(\Phi_I))$, $I\subset [M]$. One of the key tools will be Farkas' lemma
(see, e.g., \cite[Lemma  1.2.5]{matou02}).

\begin{lem}[Farkas' Lemma]\label{l:farkas}
For every real $N\times M$-matrix $A$ exactly one of the following cases occurs:
\begin{itemize}
\item[(i)] The system of linear equations $Ax=0$ has a nontrivial nonnegative solution $x\in\R^{M}$ {\rm (}i.e.,
all components of $x$ are nonnegative and at least one of them is strictly positive.{\rm )}
\item[(ii)] There exists $y\in\R^N$ such that $A^{T}y$ is a vector with all entries strictly positive.
\end{itemize}
\end{lem}

In our  first main result we use the notation $\co(A)$ for a matrix $A$ which we simply define as the
convex hull of the set of column vectors of $A$.

\begin{thm}\label{poly}
Let $M\geq m\geq N\geq 2$, and let $\Phi = \{\varphi_k\}_{k=1}^M$ be a frame for $\R^N$. Then the following
statements are equivalent:
\begin{enumerate}
\item[{\rm (i)}] $\Phi$ is $m$-scalable, respectively, strictly $m$-scalable,
\item[{\rm (ii)}] There exists a subset $I\subset [M]$ with $\# I=m$ such that $0\in\co(F(\Phi_{I}))$,  respectively, $0\in ri\co(F(\Phi_{I}))$.
\item[{\rm (iii)}] There exists a subset $I\subset [M]$ with $\# I = m$ for which there is no $h\in\R^d$ with $\<F(\vphi_k),h\> > 0$ for all $k\in I$, respectively, with $\ip{F(\varphi_k)}{h}\geq 0$ for all $k\in I$, with at least one of the inequalities being strict.
\end{enumerate}
\end{thm}
\begin{proof}
(i)$\Slra$(ii). This equivalence follows directly if we can show the following equivalences for $\Psi\subset\Phi$:
\begin{align}
\begin{split}\label{e:equis}
0\in\co(F(\Psi))&\;\Llra\;\ker F(\Psi)\setminus\{0\}\text{ contains a nonnegative vector and}\\
0\in ri\co(F(\Psi))&\;\Llra\;\ker F(\Psi)\text{ contains a positive vector.}
\end{split}
\end{align}
The implication "$\Sra$" is trivial in both cases. For the implication "$\Sla$" in the first case let
$I\subset [M]$ be such that $\Psi = \Phi_I$, $I = \{i_1,\ldots,i_m\}$, and let $u = (c_1,\ldots,c_m)^T\in\ker F(\Psi)$
be a non-zero nonnegative vector. Then $A := \sum_{k=1}^mc_k > 0$ and with $\la_k := c_k/A$, $k\in [m]$, we have
$\sum_{k=1}^m\la_k = 1$ and $\sum_{k=1}^m\la_kF(\vphi_{i_k}) = A^{-1}F(\Psi)u = 0$. Hence $0\in\co(F(\Psi))$. The
proof for the second case is similar.

(ii)$\Slra$(iii). In the non-strict case this follows from \eqref{e:equis} and Lemma \ref{l:farkas}. In the strict
case this is a known fact; e.g., see \cite[Lemma 3.6.5]{web94}.
\end{proof}

We now derive a few consequences of the above theorem. A given vector $v\in\R^d$ defines a hyperplane by
$$
H(v) = \{y\in\R^d: \ip{v}{y} = 0\},
$$
which itself determines two open convex cones $H^-(v)$ and $H^+(v)$, defined by
\begin{equation*}%\label{negcone}
H^-(v) = \{y\in\R^d : \ip{v}{y} < 0\}
\quad\text{and}\quad
H^{+}(v) = \{y\in\R^d : \ip{v}{y} > 0\}.
\end{equation*}
Using these notations we can restate the equivalence (i)$\Slra$(iii) in Theorem~\ref{poly} as follows:

\begin{prop}\label{p:polycoro}
Let $M\geq N\geq 2$, and let $m$ be such that $N\leq m\leq M$. Then a frame $\Phi = \{\varphi_k\}_{k=1}^M$ for
$\R^N$ is $m$-scalable if and only if there exists a subset $I\subset [M]$ with $\#I = m$ such that
$\bigcap_{i\in I} H^{+}(F(\varphi_i)) = \emptyset.$
\end{prop}

\begin{rem}
In the case of strict $m$-scalability we have the following necessary condition: If $\Phi$ is strictly $m$-scalable,
then there exists a subset $I\subset [M]$ with $\#I = m$ such that  $\bigcap_{i\in I} H^-(F(\varphi_i)) = \emptyset.$
\end{rem}

\begin{rem}
When $M\geq d+1 = N(N+1)/2$, we can use properties of the convex sets $H^{\pm}(F(\varphi_k))$ to give an alternative
proof of Corollary~\ref{highredund}. For this, let the frame $\Phi = \{\varphi_k\}_{k=1}^M$ for $\R^N$ be scalable.
Then, by Proposition \ref{p:polycoro} we have that $\bigcap_{k=1}^MH^{+}(F(\varphi_k)) = \emptyset$. Now, Helly's
theorem (see, e.g., \cite[Theorem 1.3.2]{matou02}) implies that there exists $I\subset [M]$ with $\# I = d+1$ such
that $\bigcap_{i\in I}H^{+}(F(\varphi_i)) = \emptyset$. Exploiting Proposition \ref{p:polycoro} again, we conclude
that $\Phi$ is $(d+1)$-scalable.
\end{rem}

The following result is an application of Proposition \ref{p:polycoro} which provides a simple condition for $\Phi\notin\sca(M,N)$.

\begin{prop}\label{firstoctant}
Let $\Phi = \{\varphi_k\}_{k=1}^M$ be a frame for $\R^N$, $N\ge 2$. If there exists an isometry $T$ such that
$T(\Phi)\subset\R^{N-2}\times\R^{2}_{+}$, then $\Phi$ is not scalable. In particular, $\Phi$ is not scalable if there
exist $i,j\in [N]$, $i\neq j$, such that $\vphi_k(i)\vphi_k(j) > 0$ for all $k\in [M]$.
\end{prop}
\begin{proof}
Without loss of generality, we may assume that $\Phi\subset\R^{N-2}\times\R^{2}_{+}$, cf.\ \cite[Corollary 2.6]{kopt12}.
Let $\{e_k\}_{k=1}^{d}$ be the standard ONB for $\R^d$. Then for each $k\in [M]$ we have that
$$
\ip{e_d}{F(\varphi_{k})} = \varphi_{k}(N-1)\varphi_{k}(N) >0.
$$
Hence, $e_d\in\bigcap_{i\in [M]}H^+(F(\vphi_i))$. By Proposition \ref{p:polycoro}, $\Phi$ is not scalable.
\end{proof}

The characterizations stated above can be recast in terms of the convex cone $C(F(\Phi))$ generated by $F(\Phi)$. We
state this result for the sake of completeness. But first, recall that for a finite subset $X = \{x_1,\ldots,x_M\}$
of $\R^d$ the \emph{polyhedral cone generated by $X$} is the closed convex cone $C(X)$ defined by
$$
C(X) = \left\{\sum_{i=1}^{M} \alpha_{i} x_i : \alpha_i\geq 0\right\}.
$$
Let $C$ be a cone in $\R^d$. The \emph{polar cone} of $C$ is the  closed convex cone $C^{\circ}$ defined by
$$
C^{\circ} := \{x\in\R^N : \ip{x}{y}\leq 0\,\,{\textrm for \, all}\,\,y\in C\}.
$$
The cone $C$ is said to be \emph{pointed} if $C\cap (-C) = \{0\},$ and \emph{blunt} if the linear space generated
by $C$ is $\R^N$, i.e.\ $\linspan C = \R^N$.

\begin{cor}\label{mainNd1}
Let $\Phi = \{\varphi_k\}_{k=1}^M\in\calF^*$, and let $N\leq m\leq M$ be fixed. Then the following conditions are equivalent:
\begin{enumerate}
\item[{\rm (i)}]   $\Phi$ is strictly $m$-scalable .
\item[{\rm (ii)}]  There exists $I \subset [M]$ with $\#I = m$ such that $C(F(\Phi_{I}))$ is not pointed.
\item[{\rm (iii)}] There exists $I \subset [M]$ with $\#I = m$ such that $C(F(\Phi_{I}))^{\circ}$ is not blunt.
\item[{\rm (iv)}]  There exists $I \subset [ M]$ with $\#I = m$ such that the interior of $C(F(\Phi_{I}))^{\circ}$ is  empty.
\end{enumerate}
\end{cor}
\begin{proof}
(i)$\Slra$(ii). By Proposition \ref{p:F}, $\Phi$ is strictly $m$-scalable if and only if there exist $I\subset [M]$
with $\#I = m$ and a nonnegative $u\in\ker F(\Phi_I)\setminus\{0\}$ with $\|u\|_0 = m$. By~\cite[Lemma 2.10.9]{stwi70},
this is equivalent to the cone $C(F(\Phi_I))$ being not pointed. This proves that (i) is equivalent to (ii).

(ii)$\Slra$(iii). This follows from the fact that the polar of a pointed cone $C$ is blunt and vice versa, as long as
$C^{\circ \circ}=C$, see \cite[Theorem 2.10.7]{stwi70}. But in our case $C(F(\Phi_I))^{\circ \circ}=C(F(\Phi_I))$, see
\cite[Lemma 2.7.9]{stwi70}.

(iii)$\Sra$(iv). If $C(F(\Phi_I))^\circ$ is not blunt, then it is contained in a proper hyperplane of $\R^d$ whose
interior is empty. Hence, also the interior of $C(F(\Phi_I))^\circ$ must be empty.

(iv)$\Sra$(iii). We use a contra positive argument. Assume that $C(F(\Phi))^{\circ}$ is blunt. This is equivalent
to $\linspan C(F(\Phi))^{\circ} = \R^d$. But for the nonempty cone $C(F(\Phi))^{\circ}$ we have $\aff(C(F(\Phi))^{\circ})
= \linspan C(F(\Phi))^{\circ}$. Hence, $\aff(C(F(\Phi))^{\circ})=\R^d$, and so the relative interior of $C(F(\Phi))^{\circ}$
is equal to its interior, which therefore is nonempty.
\end{proof}

The main idea of the previous results is the characterization of ($m$-)scalability of $\Phi$ in terms of properties of
the convex polytopes $\co(F(\Phi_I))$. However, it seems more ``natural'' to seek  assumptions on the convex polytopes $\co(\Phi_I)$ that will ensure that $\co(F(\Phi))$ satisfy  the conditions in Theorem~\ref{poly} 
%and some of its corollaries 
hold. 
%This is a very interesting  question that we have not  addressed here, but that we shall investigate elsewhere. 
Proposition~\ref{firstoctant}, which gives a condition on $\Phi$ that precludes it to be scalable, is a step in this direction.

Nonetheless, we address the related question of whether $F(\Phi)$ is a frame for $\R^d$ whenever $\Phi$ is a scalable frame for
$\R^N$. This depends clearly on the redundancy of $\Phi$ as well as on the map $F$. In particular, we finish this section
by giving a condition which ensures that  $F(\Phi)$ is always a frame for $\R^d$ when $M\geq d+1$. In order to prove this
result, we need a few preliminary facts.

For $x=(x_k)_{k=1}^{N}\in \R^N$ and $h=(h_k)_{k=1}^d \in \R^d$, we have that
\begin{equation}\label{poly2}
\ip{F(x)}{h} = \sum_{\ell=2}^{N}h_{\ell-1}(x_1^2 - x_{\ell}^2) + \sum_{k=1}^{N-1}\sum_{\ell=k+1}^{N}h_{k(N - 1 - (k-1)/2) +\ell - 
1}x_{k}x_{\ell}.
\end{equation}
The right hand side of \eqref{poly2} is obviously a homogeneous polynomial of degree $2$ in $x_1, x_2, \hdots, x_N$%and consists of monomials of degree $2$ as well
. We shall denote the set of all polynomials of this form by $\boldsymbol{P}_{2}^N$. It is easily seen that $\boldsymbol{P}_{2}^N$ is isomorphic to the  subspace of real symmetric $N\times N$ matrices whose trace is $0$. Indeed, for each $N\geq 2$, and each $p\in\boldsymbol{P}_{2}^N$,
$$
p(x) = \sum_{\ell=2}^{N}a_{\ell-1}(x_1^2 - x_{\ell}^2) + \sum_{k=1}^{N-1}\sum_{\ell=k+1}^{N}a_{k(N - (k+1)/2) +\ell - 1}x_{k}x_{\ell},
$$
we have $p(x) = \<Q_px,x\>$, where $Q_p$ is the symmetric $N\times N$-matrix with entries
$$
Q_{p}(1,1) = \sum_{k=1}^{N-1}a_k, \qquad Q_{p}(\ell,\ell) = -a_{\ell-1}\quad\text{for }\ell = 2,3,\ldots,N
$$
and
$$
Q_{p}(k,\ell) = \frac 1 2 a_{k(N - (k+1)/2) +\ell - 1}\quad\text{for }k = 1,\ldots,N-1,\;\ell = k+1,\ldots,N.
$$
In particular, the dimension of $\boldsymbol{P}_{2}^N$ is $d=(N+2)(N-1)/2$.
%
%Using this fact and Proposition~\ref{p:less_vectors}, we show that when $M\geq d+1$, and if $\Phi \in \sca_{+}(d+1)$, then
%$\co(F(\Phi))$ is a ``full'' polytope. We shall also need the following version of Farkas' lemma:
%
%\begin{lem}[{\cite[Corollary 22.3.1]{rocka}}]
%For every real $N\times M$-matrix $A$ and every $b\in\R^N\setminus\{0\}$ the following two conditions are equivalent:
%\begin{itemize}
%\item[(i)] The system of linear equations $Ax=b$ has a nonnegative solution $x\in\R^{M}$.
%\item[(ii)] There exists no  $y \in \R^N$ such that $y ^{T}A\geq 0$ and $\ip{y}{b}<0$.
%\end{itemize}
%\end{lem}

\begin{prop}\label{dimpoly}
Let $M\geq d+1$ where $d=(N-1)(N+2)/2$, and $\Phi = \{\varphi_k\}_{k=1}^M\in\sca_{+}(d+1)\setminus\sca(d)$. Then 
$F(\Phi)$ is a frame for $\R^d$.
\end{prop}
\begin{proof}
%Since $\Phi$ is strictly $(d+1)$-scalable, there exists $\Phi_{I}\subset \Phi$ with $\#I=d+1$ such that for a strictly
%positive vector $x\in \R^{d+1}$ such that $F(\Phi_{I})x=0$. Hence, $0\in \co(F(\Phi_{I}))$. We only need to prove that
%$\co(F(\Phi_{I}))$ contains a linearly independent set consisting of $d$ elements.
%
%Let $\{e_k\}_{k=1}^{d}$ be the standard orthonormal basis of $\R^d$. We claim that for each $k=1, 2, \hdots, d$
%$e_k=F(\Phi_{I})x_k$ for some nonnegative vector $x_k\in \R^{d+1}$. We prove this using Farkas's lemma (Lemma B).
Let $I\subset [M]$, $\#I=d+1$, be an index set such that  $\Phi_I$ is strictly scalable. Assume that there exists $h\in\R^d$ such that $\<F(\varphi_k),h\> = 0$ for each $k\in I$. By \eqref{poly2} we conclude that $p_h(\vphi_k) = 0$ for all $k\in I$, where $p_h$ is the polynomial in $\boldsymbol{P}_{2}^N$ on the right hand side of \eqref{poly2}. Hence $\<Q_{p_h}\vphi_k,\vphi_k\> = 0$ for all $k\in I$. Now, we have
\begin{equation}\label{HS}
\<\vphi_k\vphi_k^T,Q_{p_h}\>_{HS} = \Tr(\vphi_k\vphi_k^TQ_{p_h}) = \<Q_{p_h}\vphi_k,\vphi_k\> = 0\quad\text{for all }k\in I.
\end{equation}
But as $\Phi_I$ is not $d$-scalable (otherwise, $\Phi\in\sca(d)$) it is not $m$-scalable for every $m\le d$. Thus, Proposition \ref{p:less_vectors} yields that
$$
\dim\linspan\{\vphi_k\vphi_k^T : k\in I\} = d+1.
$$
Equivalently, $\{\vphi_k\vphi_k^T : k\in I\}$ is a basis of the $(d+1)$-dimensional space $S_N$. Therefore, from \eqref{HS} we conclude that $Q_{p_h} = 0$ which implies $p_h = 0$ (since $p\mapsto Q_p$ is an isomorphism) and thus $h = 0$.

Now, it follows that $F(\Phi_I)$ spans $\R^d$ which is equivalent to $F(\Phi_I)$ being a frame for $\R^d$. Hence, so is $F(\Phi)$.
%But, by Proposition~\ref{p:less_vectors}  $\{\varphi_{k}\varphi_{k}^{T}\}_{k=1}^M$  form a complete set.
%Hence, $Q_y=0$ but $y\neq 0$. But this is again a contradiction. Hence, such a $y$ cannot exists and thus for each $k\in I$,
%$e_k=F(\Phi_{I})x_k$ for some nonnegative vector $x\in \R^{d+1}$. Hence, $\{e_k\}_{k=1}^d \cup \{0\}$ is an affinely
%independent set in $\co(F(\Phi_I)\subset \co(F(\Phi))$.
\end{proof}

%**********************************************************************************************************************
%**********************************************************************************************************************
%**********************************************************************************************************************

\section{Topology of the Set of Scalable Frames}
\label{sec3}

In this section, we present some topological features of the set $\sca(M,N)$. Hereby, we identify frames in $\calF(M,N)$
with real $N\times M$-matrices as we already did before, see, e.g., \eqref{e:scfrm} in subsection \ref{subsec2.3}. Hence,
we consider $\calF(M,N)$ as the set of all matrices in $\R^{N\times M}$ of rank $N$. Note that under this identification
the order of the vectors in a frame becomes important. However, it allows us to endow $\fr(M,N)$ with the usual Euclidean
topology of $\R^{N\times M}$.

In \cite{kopt12} it was proved that $\sca(M,N)$ is a closed set in $\calF(M,N)$ (in the relative topology of $\calF(M,N)$).
The next proposition refines this fact.

\begin{prop}\label{p:closedness}
Let $M\geq m\geq N\geq 2$. Then $\sca(M,N,m)$ is closed in $\calF(M,N)$.
\end{prop}

\begin{proof}
We prove the assertion by establishing that the complement $\calF\setminus\sca(m)$ is open, that is, if
$\Phi = \{\varphi_{k}\}_{k=1}^{M}\in \fr$ is a frame which is not $m$-scalable, we prove that there exists $\varepsilon>0$
such that for any collection $\Psi = \{\psi_{k}\}_{k=1}^{M}$ of vectors in $\R^N$ for which
$$
\|\varphi_k - \psi_k\| < \varepsilon\quad\text{for all $k\in [M]$,}
$$
we have that $\Psi$ is a frame which is not $m$-scalable. Thus assume that $\Phi = \{\varphi_k\}_{k=1}^{M}$ is a frame
which is not $m$-scalable and define the finite set $\calI$ of subsets by
$$
\calI := \{I\subset [M] : \#I = m\}.
$$
By Proposition \ref{p:polycoro}, for each $I\in\calI$ there exists $y_I\in\bigcap_{k\in I}H^{+}(F(\varphi_k))$, that is,
$\min_{k\in I}\ip{y_I}{F(\varphi_k)} > 0$. By the continuity of the map $F$, there exists $\varepsilon > 0$ such that
for each $\{\psi_k\}_{k=1}^M\subset\R^N$ with $\|\psi_k - \varphi_k\| < \varepsilon$ for all $k\in [M]$ we still have
$\min_{k\in I}\ip{y_{I}}{F(\psi_k)} > 0$. We can choose $\varepsilon>0$ sufficiently small to guarantee that $\Psi =
\{\psi_k\}_{k=1}^{M} \in \fr$. Again from Proposition \ref{p:polycoro} we conclude that $\Psi$ is not $m$-scalable for
any $N\leq m\leq M$. Hence, $\sca(m)$ is closed in $\fr$.
\end{proof}

The next theorem is the main result of this section. It shows that the set of scalable frames is nowhere dense in the
set of frames unless the redundancy of the considered frames is disproportionately large.

\begin{thm}\label{bdMlessd}
Assume that $2\le N\le M < d+1 = N(N+1)/2$. Then $\sca(M,N)$ does not contain interior points. In other words, for the
boundary of $\sca(M,N)$ we have
$$
\partial\sca(M,N) = \sca(M,N).
$$
\end{thm}

For the proof of Theorem \ref{bdMlessd} we shall need two lemmas. Recall that for a frame $\Phi = \{\varphi_k\}_{k=1}^M\in\fr$
we use the notation
$$
X_{\Phi} = \{\varphi_i\varphi_i^T : i\in [M]\}.
$$

\begin{lem}\label{l:not_all}
Let $\{\varphi_k\}_{k=1}^M\subset\R^N$ be such that $\dim\linspan X_{\Phi} < \frac{N(N+1)}{2}$. Then there exists
$\varphi_0\in\R^N$ with $\|\varphi_0\| = 1$ such that $\varphi_0\varphi_0^T\notin\linspan X_\Phi $.
\end{lem}

\begin{proof}
Assume the contrary. Then each rank-one orthogonal projection is an element of $\linspan X_\Phi $. But by the spectral
decomposition theorem every symmetric matrix in $\R^{N\times N}$ is a linear combination of such projections. Hence,
$\linspan X_{\Phi} $ coincides with the linear space $S_N$ of all symmetric matrices in $\R^{N\times N}$. Therefore,
$$
\dim\linspan X_{\Phi}  = \frac{N(N+1)}{2},
$$
which is a contradiction.
\end{proof}

The following lemma shows that for a generic $M$-element set $\Phi=\{\varphi_i\}_{i=1}^M\subset\R^N$ (or matrix in $\R^{N\times M}$,
if the $\varphi_i$ are considered as columns) the subspace $\linspan X_{\Phi}$ has the largest possible dimension.

\begin{lem}\label{l:dense}
Let $D := \min\{M,N(N+1)/2\}$. Then the set
$$
\big\{\Phi\in\R^{N\times M} : \dim\linspan X_{\Phi} = D\big\}
$$
is dense in $\R^{N\times M}$.
\end{lem}

\begin{proof}
Let $\Phi = \{\varphi_i\}_{i=1}^M\in\R^{N\times M}$ and $\varepsilon > 0$. We will show that there exists $\Psi =
\{\psi_i\}_{i=1}^M\in\R^{N\times M}$ with $\|\Phi - \Psi\| < \varepsilon$ and $\dim\linspan X_{\Psi}  = D$. For this, set
$\mathcal{W} := \linspan X_{\Phi}$ and let $k$ be the dimension of $\mathcal{W}$. If $k=D$, nothing is to prove. Hence,
let $k < D$. Without loss of generality, assume that $\varphi_1\varphi_1^T,\ldots,\varphi_k\varphi_k^T$ are linearly
independent. By Lemma \ref{l:not_all} there exists $\varphi_0\in\R^N$ with $\|\varphi_0\| = 1$ such that
$\varphi_0\varphi_0^T\notin\mathcal{W}$. For $\delta > 0$ define the symmetric matrix
$$
S_\delta := \delta\left(\varphi_{k+1}\varphi_0^T + \varphi_0\varphi_{k+1}^T\right) + \delta^2\varphi_0\varphi_0^T.
$$
Then there exists {\em at most} one $\delta > 0$ such that $S_\delta\in\mathcal{W}$ (regardless of whether
$\varphi_{k+1}\varphi_0^T + \varphi_0\varphi_{k+1}^T$ and $\varphi_0\varphi_0^T$ are linearly independent or not).
Therefore, we find $\delta > 0$ such that $\delta < \varepsilon /M$ and $S_\delta\notin\mathcal{W}$. Now, for $i\in [M]$ put
$$
\psi_i :=
\begin{cases}
\varphi_i &\text{if }i\neq k+1\\
\varphi_{k+1} + \delta\varphi_0 &\text{if }i=k+1
\end{cases}
$$
and $\Psi := \{\psi_i\}_{i=1}^M$. Let $\lambda_1,\ldots,\lambda_{k+1}\in\R$ such that
$$
\sum_{i=1}^{k+1}\lambda_i\psi_i\psi_i^T = 0.
$$
Then, since $\psi_{k+1}\psi_{k+1}^T = \varphi_{k+1}\varphi_{k+1}^T + S_\delta$, we have that
$$
\lambda_{k+1}S_\delta = -\sum_{i=1}^{k+1}\lambda_i\varphi_i\varphi_i^T\in\mathcal{W},
$$
which implies $\lambda_{k+1} = 0$ and therefore also $\lambda_1 = \ldots = \lambda_k = 0$. Hence, we have
$\dim\linspan X_\Psi = k+1$ and $\|\Phi - \Psi\| < \varepsilon/M$. If $k = D-1$, we are finished. Otherwise, repeat
the above construction at most $D - k - 1$ times.
\end{proof}

\begin{rem}
For the case $M\ge N(N+1)/2$, Lemma {\rm\ref{l:dense}} has been proved in {\rm\cite[Theorem 2.1]{cc13}}. In the
proof, the authors note that $X_\Phi$ spans $S_N$ if and only if the frame operator of $X_\Phi$ {\rm (}considered
as a system in $S_N${\rm )} is invertible. But the determinant of this operator is a polynomial in the entries of
$\vphi_i$, and the complement of the set of roots of such polynomials is known to be dense.
\end{rem}

{\sc  Proof of Theorem }\ref{bdMlessd}.
Assume the contrary. Then, by Lemma \ref{l:dense}, there even exists an interior point $\Phi = \{\varphi_i\}_{i=1}^M\in\sca(M,N)$
of $\sca(M,N)$ for which the linear space
$
\mathcal{W} := \linspan X_{\Phi} $
has dimension $M$. Since $\Phi$ is scalable, there exist $c_1,\ldots,c_M\ge 0$ such that
$$
\sum_{i=1}^M c_i\varphi_i\varphi_i^T = \Id.
$$
Without loss of generality we may assume that $c_1 > 0$.

By Lemma \ref{l:not_all} there exists $\varphi_0\in\R^N$ with $\|\varphi_0\| = 1$ such that $\varphi_0\varphi_0^T\notin\mathcal{W}$.
As in the proof of Lemma \ref{l:dense}, we set
$$
S_\delta := \delta\left(\varphi_1\varphi_0^T + \varphi_0\varphi_1^T\right) + \delta^2\varphi_0\varphi_0^T.
$$
Then, for $\delta > 0$ sufficiently small, $S_\delta\notin\mathcal{W}$ and $\Psi := \{\varphi_1 + \delta\varphi_0,\varphi_2,
\ldots,\varphi_M\}\in\sca(M,N)$. Hence, there exist $c_1',\ldots,c_M'\geq 0$ such that
$$
\sum_{i=1}^M c_i\varphi_i\varphi_i^T = \Id = c_1'(\varphi_1 + \delta\varphi_0)(\varphi_1 + \delta\varphi_0)^T
+ \sum_{i=2}^Mc_i'\varphi_i\varphi_i^T = \sum_{i=1}^Mc_i'\varphi_i\varphi_i^T + c_1'S_\delta.
$$
This implies $c_1'S_\delta\in\mathcal{W}$, and thus $c_1' = 0$. But then we have
$$
c_1\varphi_1\varphi_1^T + \sum_{i=2}^M(c_i - c_i')\varphi_i\varphi_i^T = 0,
$$
which yields $c_1 = 0$ as the matrices $\varphi_1\varphi_1^T,\ldots,\varphi_M\varphi_M^T$ are linearly independent. A contradiction.\hfill\qed

%**********************************************************************************************************************
%**********************************************************************************************************************
%**********************************************************************************************************************

\section*{ACKNOWLEDGMENTS}

G.~Kutyniok acknowledges support by the Einstein Foundation Berlin, by Deutsche Forschungsgemeinschaft
(DFG) Grant KU 1446/14, and by the DFG Research Center {\sc Matheon} ``Mathematics for key technologies''
in Berlin. F.~Philipp is supported by the DFG Research Center {\sc Matheon}. K.~A.~Okoudjou  was supported
by ONR grants N000140910324 and N0001\-40910144, by a RASA from the Graduate School of UMCP and by the
Alexander von Humboldt foundation. He would also like to express his gratitude to the Institute for
Mathematics at the Universit\"at Osnabr\"uck and the Institute of Mathematics at the Technische Universit\"at
Berlin for their hospitality while part of this work was completed.

%**********************************************************************************************************************
%**********************************************************************************************************************
%**********************************************************************************************************************

\end{document}